\documentclass[2pt]{article}

\usepackage{lipsum} 
\usepackage{placeins} 
\usepackage[sc]{mathpazo} 
\usepackage[T1]{fontenc} 
\usepackage{microtype} 
\usepackage{indentfirst} 
\usepackage{geometry} 
\usepackage{multicol} 
\usepackage[hang, small,labelfont=bf,up,textfont=it,up]{caption} 
\usepackage{booktabs} 
\usepackage{float} 
\usepackage{hyperref} 
\usepackage{amsmath,amssymb}
\usepackage{graphicx}
\usepackage{listings,xcolor}

\usepackage{color}

\usepackage{mathrsfs}

\usepackage{lastpage}
\usepackage{fancyhdr}
\usepackage{framed}
\usepackage{algorithm2e}
\numberwithin{equation}{section}

\newcommand{\bbr}[1]{\left[ {#1} \right]}
\newcommand{\bpa}[1]{\left( {#1} \right)}
\newcommand{\bst}[1]{\left\{ {#1} \right\}}
\newcommand{\bag}[1]{\left\langle {#1} \right\rangle}

\newcommand{\cm}{,\;}

\newcommand{\abs}[1]{\left| {#1} \right|}

\newcommand{\eqal}[1]{\begin{equation}\begin{aligned} #1 \end{aligned}\end{equation}}
\newcommand{\eqals}[1]{\begin{equation*}\begin{aligned} #1 \end{aligned}\end{equation*}}

\newcommand{\idx}[2]{{#1}^{({#2})}}
\newcommand{\mb}[1]{\mathbb{{#1}}}

\newcommand{\wt}{\widetilde}

\usepackage{lettrine} 
\usepackage{paralist} 
\usepackage{amssymb,amsmath,bm}

\usepackage{subcaption}
\usepackage{abstract} 

\usepackage{titlesec} 
\renewcommand\thesection{\Roman{section}} 
\renewcommand\thesubsection{\Alph{subsection}} 
\renewcommand\thesubsubsection{\arabic{subsubsection}}
\titleformat{\section}[block]{\Large\scshape\centering}{\thesection.}{2em}{} 
\titleformat{\subsection}[block]{\large\scshape\centering}{\thesection.\thesubsection.}{1em}{} 
\titleformat{\subsubsection}[block]{\scshape\centering}{\thesubsection.\thesubsubsection}{1em}{} 

\usepackage{amsthm}

\newtheorem*{theorem*}{Theorem}
\newtheorem{theorem}{Theorem}[section]
\newtheorem{assumption}{Assumption}[section]
\newtheorem*{lemma*}{Lemma}
\newtheorem{lemma}{Lemma}[section]
\newtheorem*{proposition*}{Proposition}

\newtheorem*{remark*}{Remark}
\newtheorem{remark}{Remark}[section]

\theoremstyle{definition}

\usepackage{titlesec}

\title{\vspace{0mm}\fontsize{14pt}{1pt}\selectfont\textbf{On the convergence of BFGS on a class of piecewise linear non-smooth functions}} 

\author{
\textsc{Yuchen Xie}\thanks{Department of Industrial Engineering and Management Sciences, Northwestern University, Evanstion IL 60208, USA. \href{mailto:ycxie@u.northwestern.edu}{ycxie@u.northwestern.edu}} \quad
\textsc{Andreas W\"achter}\thanks{Department of Industrial Engineering and Management Sciences, Northwestern University, Evanstion IL 60208, USA. \href{mailto:andreas.waechter@u.northwestern.edu}{andreas.waechter@northwestern.edu}}
}

\begin{document}

\nointerlineskip  
\let\snewpage \newpage
\let\newpage \relax
\maketitle
\let \newpage \snewpage

\thispagestyle{fancy}


%
\begin{abstract}

\noindent The quasi-Newton Broyden--Fletcher--Goldfarb--Shanno (BFGS) method has proven to be very reliable and efficient for the minimization of smooth objective functions since its inception in the 1960s.  Recently, it was observed empirically that it also works remarkably well for non-smooth problems when combined with the Armijo-Wolfe line search, but only very limited theoretical convergence theory could be established so far.  In this paper, we expand these results by considering convex piecewise linear functions with two pieces that are not bounded below.  We prove that the algorithm always terminates in a finite number of iterations, eventually generating an unbounded direction.  In other words, in contrast to the gradient method, the BFGS algorithm does not converge to a non-stationary point.

\end{abstract}
%
%

\section{Introduction}\label{sec:introduction}
Quasi-Newton methods are a class of optimization algorithms for the minimization of smooth objective functions. 
They utilize changes in the gradient between iterations to construct quadratic models and maintain the fast local convergence speed of Newton's method when the objective function is smooth, without the need of second order derivatives. Among quasi-Newton methods, the Broyden-Fletcher-Goldfarb-Shanno (BFGS) update 
has been one of the most successful algorithms for smooth optimization problems. 

Since quasi-Newton methods involve approximations to the Hessian of the function under consideration, one might expect that they would not work for non-differentiable objective functions for which the Hessian does not exist at all points. Surprisingly, Lewis and  Overton \cite{lewis2013nonsmooth} observed that in practice the BFGS method, paired with the Armijo-Wolfe line search, converges reliably to local minimizers of non-smooth problems. 
In addition, they observed that the Hessian approximation 
 becomes increasingly ill-conditioned when the method converges to a minimizer at which the function is not differentiable.  
%
In that case, the eigenvalues along directions corresponding to the U-subspace, the manifold in which the function is differentiable, remain finite, while the directions corresponding to the orthogonal V-subspace grow unbounded. 
 This indicates that the BFGS update is able to approximate the infinite curvature of a non-smooth function and makes use of this knowledge to facilitate the minimization process.


Despite the empirical success of the BFGS method on non-smooth  problems, its theoretical convergence properties remain elusive and only very limited convergence theory could be established so far. For smooth convex functions, the BFGS method has been shown to converge to minimizers \cite{powell1976some}, with a local super-linear rate when the objective is strongly convex \cite{byrd1989tool}. On the other hand, counter examples of non-convex smooth functions exists for which the BFGS methods does not converge \cite{dai2002convergence,mascarenhas2004bfgs,dai2013perfect}.

However, for the non-smooth case, establishing a general convergence proof as well as a convergence rate analysis remains an open question. There have been only a couple of publications which consider very special classes of non-smooth functions.  
Lewis and  Overton \cite{lewis2013nonsmooth} showed that BFGS with exact line search converges globally to the minimizer for the $\ell_2$-norm function in the $2$-dimensional space. 
More recently, Guo and Lewis \cite{guo2017bfgs} showed that the BFGS method with Armijo-Wolfe line search converges globally to the minimizer for convex functions if the minimizer is unique and  the only point of non-differentiability.  They combine Powell's original proof \cite{powell1976some} for the smooth and convex case with an argument which uses a sequence of convex smooth functions that differ from the non-smooth function only in a shrinking neighborhood of the non-smooth point.  
In particular, this extends the result of \cite{lewis2013nonsmooth} for the $\ell_2$-norm objective by permitting an inexact Armijo-Wolfe line search and dimensions $n$ larger than 2. However, their argument cannot be extended to non-smooth functions with non-isolated points of non-differentiability, such as the $\ell_1$-norm in $n$-dimensional space.
On the other hand, there have been attempts to modify the BFGS method so that convergence can be guaranteed, such as \cite{curtis2015quasi}, but these methods do not address the convergence for the unaltered BFGS method.

While it remains unknown how to prove convergence of the unadulterated BFGS method for general convex non-smooth functions, in this paper we 
analyze the special case of convex piecewise linear functions with only two pieces that are not bounded below. 
For such functions, the points of non-differentiability lie on a straight line of dimension 1, which makes it characteristically different from the functions with isolated points of non-differentiability analyzed in \cite{guo2017bfgs}.  While being very simple, this class of functions demonstrates a fundamental difference between the gradient method and the BFGS method.  For a particular instance of this class, Asl and Overton \cite{asl2017analysis} showed that the gradient descent method with line search generates a sequence of iterates converging to a non-stationary point on the line of points of non-differentiability.  
In contrast to this, we give a proof that the BFGS method will never generate iterates converging to a point of non-differentiability for these functions.  
Instead, we show that it always terminates in a finite number of iterations, eventually generating a search direction in which the objective is unbounded below, driving the function value to $-\infty$.

One immediate consequence of our result is that when the BFGS algorithm is applied to a piecewise linear continuous function, the iterates cannot converge to a non-stationary point $x_*$ that lies in manifold of non-differentiability that has dimension at most 1; see Remark~\ref{rem_generalized}.  For example, for $f(x)=\|x\|_1$, the method cannot converge to a limit point at which exactly one of the component $\idx{x}{i}_*$ is zero.
Of course, one would like to prove much stronger convergence results; for example, that any limit point of the iterates of the algorithm is a stationary point, for general non-smooth objective functions that are bounded below.
%
We hope that this paper takes another illuminating step towards this goal by offering proof techniques that are fundamentally different from those in \cite{lewis2013nonsmooth} and \cite{guo2017bfgs}.
%


The paper is organized as follows. After stating the problem in Section~\ref{sec:problem}, we define some iterative quantities that are useful for our analysis in Section~\ref{sec:iter_relations}. In Section~\ref{sec:mainresults} we prove our main results. In Section~\ref{sec:numerical} we provide some numerical results, examining the quantities defined in our theoretical analysis. We finish with some concluding remarks in Section~\ref{sec:conclude}.

\textbf{Notation.}
The $i$-th component of a vector $x\in\mb{R}^n$ is denoted by $\idx{x}{i}$.  For $K\in\mb{N}^+$ we define the set $[K]=\{0,1,2,\ldots,K\}$.


\section{Description of the problem}\label{sec:problem}

We are concerned with the standard BFGS method combined with the Armijo-Wolfe line search, when it is applied to the problem
\eqal{
	\label{eq_problem_definition}
	\min_{x \in \mb{R}^n} \; f(x) = \abs{\idx{x}{1}} + \sum_{i = 2}^n \idx{x}{i} \cm (n \in \mb{N}^+, n \geq 2).
}


Let $\{x_0, x_1, x_2, ..., x_k, ...\}$ be the iterates generated by the algorithm defined as follows \cite{wright2008numerical}.  The method starts with an initial iterates $x_0$ and an initial positive definite approximation $H_0$ of the inverse Hessian. At iteration $k$, a search direction, denoted as $p_k$,  is computed by
\eqals{
	p_k = - H_k \nabla f(x_k).
}
Then, a line search method chooses a step size $\alpha_k>0$ that satisfies the Armijo-Wolfs conditions, namely,
\begin{align}
	f(x_k + \alpha p_k) & \leq f(x_k) + c_1 \alpha_k \bbr{\nabla f(x_k)}^\intercal p_k  \label{eq_suff_dec}\\
	\bbr{\nabla f(x_k + \alpha p_k)}^\intercal p_k & \geq c_2 \bbr{\nabla f(x_k)}^\intercal p_k, \label{eq_curv}
\end{align}
where $0 < c_1 < c_2 < 1$ are the parameters of the line search algorithm. We shall call these two conditions above the \textbf{sufficient decrease condition} and the \textbf{curvature conditions}, respectively. Then, the next iterate is computed by 
\eqals{
	x_{k+1} = x_k + \alpha_k p_k.
}
The inverse Hessian approximation is updated by the BFGS formula
\eqal{\label{eq_H_update}
	H_{k+1} = \bbr{I - \rho_k {s_k}{y_k}^\intercal} H_k \bbr{I - \rho_k {y_k}{s_k}^\intercal} + \rho_k s_k {s_k}^\intercal,
}
where $s_k = x_{k+1} - x_k$, $y_k = \nabla f(x_{k+1}) - \nabla f(x_k)$, and $\rho_k = \bpa{{s_k}^\intercal y_k}^{-1}$. It can be shown that if $\alpha_k$ is chosen to satisfy the Armijo-Wolfe conditions and if $H_k$ is positive definite, then $H_{k+1}$ is also positive definite. 

We make the following assumption throughout the paper.
\begin{assumption}
\label{assumption_non_smooth_points_always_avoided}
The function $f(.)$ is differentiable at all iterates generated by the BFGS algorithm. This is equivalent to saying that $\idx{x_k}{1} \neq 0$ for all $k$.  
\end{assumption}
Note, that the set of all step sizes that satisfy the Armijo-Wolfe conditions at iteration $k$, if not empty, is always an open set. Therefore, as long as any acceptable step size exists, it is always possible to find a step size such that $\idx{x}{1}_{k+1} \neq 0$. 
This assumption has also been made in \cite{lewis2013nonsmooth}.  Curtis and Que \cite{curtis2015quasi} describe a practical technique to enforce it.

Now, since the function $f(.)$ is not bounded below, it might occur that in a certain iteration the line search algorithm is unable to produce a step size $\alpha_k>0$ that satisfies the Armijo-Wolfe conditions. When this happens, we say that \textbf{the algorithm terminates}. In practice, we observe that the BFGS method applied to \eqref{eq_problem_definition} with \textbf{any} initialization always terminates in a finite number of iterations. In particular, the termination is always due to the failure of finding a step size that satisfies the curvature condition \eqref{eq_curv}. Note that when this happens, the function along the search direction $f(x_k + t p_k) \cm t > 0$ must be unbounded below, because otherwise an acceptable step size exists \cite{wright2008numerical}. Typical implementations of the Armijo-Wolfe line search  generate a sequence of trial step sizes that converges to $+\infty$ in this case, driving the function value to $-\infty$. In this paper, we aim to prove that this always happens, for any initialization and any choice of $0 < c_1 < c_2 < 1$. 

It is worth pointing out that although we are studying the simple function \eqref{eq_problem_definition}, our results can be immediately generalized to any function in the form
\eqal{\label{eq_g_affine}
	g(x) = \abs{v_1^\intercal x} + v_2^\intercal x,
}
where $v_1, v_2 \in \mb{R}^n$  are linearly independent. This is a direct consequence of the invariance of BFGS method under affine transformations of the variables. Specifically, the BFGS method has the following desirable property \cite{lewis2013nonsmooth}: 
\begin{theorem}\label{thm_equivalent}
Let $A$ be an invertible $n \times n$ matrix. Consider $g(x) = h(A x)$, where $g, h: \mb{R}^n \to \mb{R}$. Let $\bst{x_i}, \bst{\alpha_i}, \bst{H_i}$ be the iterates, step sizes, and inverse Hessian approximations generated by the BFGS method applied to $g(x)$, with initial point $x_0$ and initial inverse Hessian approximation $H_0$. Let $\bst{y_i}, \bst{\beta_i}, \bst{G_i}$ be the iterates, step sizes, and inverse Hessian approximations generated by the BFGS method applied to $h(x)$ with initial point $y_0 = A x_0$ and initial inverse Hessian approximation $G_0 = A H_0 A^\intercal$. 
Then the following holds:
\begin{itemize}
	\item If $\alpha_i = \beta_i \cm \forall i \in [k-1]$, then $y_k = A x_k$ and $G_k = A H_k A^\intercal$
	\item If $\alpha_i = \beta_i \cm \forall i \in [k-1]$, then any acceptable step size $\alpha_k$ is also an acceptable step size $\beta_k$, and vice versa.
\end{itemize}
\end{theorem}
This theorem can be verified easily using the BFGS update formula \ref{eq_H_update}. 
%
Now, consider the function $g(x)$ defined in \eqref{eq_g_affine}
with $v_1, v_2 \in \mb{R}^n$, $v_1, v_2 \neq 0$ where $v_1, v_2$ are linearly independent. Since $v_1, v_2$ are independent, they can be extended to form a basis $\bst{v_1, v_2, ..., v_n}$. Then the matrix 
\eqals{
	A = \left(
\begin{array}{c}
 v_1^\intercal \\
 \\
 v_2^\intercal-v_3^\intercal \\
 \\	
 v_3^\intercal-v_4^\intercal \\
 \\
 ... \\
 \\
 v_{n-1}^\intercal-v_n^\intercal \\
 \\
 v_n^\intercal \\
\end{array}
\right)
}
is invertible and one can see that 
\eqals{
	f(A x) = \abs{v_1^\intercal x} + \bpa{v_2 - v_3 + v_3 - v_4 + ... + v_{n-1} - v_n + v_n}^\intercal x = \abs{v_1^\intercal x} + v_2^\intercal x = g(x).
}
Therefore, we can use Theorem \ref{thm_equivalent} above to show that applying the BFGS algorithm to $g(x)$ is equivalent to applying it to $f(x)$, only with a different initialization. Hence, by considering the simple case $f(x)$, we can obtain results that are valid for any piecewise linear function with two pieces that is unbounded below.
%
%
%

\section{Iterative relations}\label{sec:iter_relations}

We note that $\nabla f(.)$ is constant in the region $\{x: \idx{x}{1} > 0\}$ and $\{x: \idx{x}{1} < 0\}$. It is then easy to see that the curvature condition \ref{eq_curv} will always ensure that 
\eqals{
	\text{sgn}(\idx{x}{1}_k) = - \text{sgn}(\idx{x}{1}_{k+1}).
}
To simplify the analysis, it is useful to ``reflect'' the iterate $x_k$ at each iteration and work with iterates $\wt{x}_k$ that always stay in $\{\wt{x}: \idx{\wt x}{1} > 0\}$. To this end, we take advantage of the fact that BFGS method is invariant to affine transformation. We define
\eqals{
	V = \text{diag}(-1, 1, 1, ..., 1) \in \mb{R}^{n \times n}
}
and let 
\eqal{
	\label{eq_definition_for_reflected_variables}
	\wt{x}_k  = \bpa{V}^k x_k \quad \text{ and } \quad
	\wt{H}_k = \bpa{V}^k H_k \bpa{V}^k.
}
Note that $V^{-1} = V = V^\intercal$. We will consider $\wt{x}_k$ and $\wt{H}_k$ instead of $x_k$ and $H_k$, so that we can make use of the nice property that $\idx{\wt{x}}{1}_k$ is always positive.

Let 
\eqals{
	\eta & = (1, 1, ..., 1)^\intercal \in \mb{R}^n \\
	\zeta & = (1, 0, ..., 0) ^\intercal \in \mb{R}^n.
}
For future reference, note that
\eqal{\label{eq_Veta}
	V\eta = -\eta \quad\text{ and }\quad
	V\eta = \eta - 2\zeta,
}
as well as
\eqal{\label{eq_grad_f}
	f(\wt x_k) = \eta^\intercal \wt x_k \quad\text{ and } \quad \nabla f(\wt x_k ) = (V)^k \eta.
}
We have the following update formula for $\wt{x}_k$ and $\wt{H}_k$:
\begin{align}	
	\wt{x}_{k+1} & = V(\wt{x}_k + \alpha_k \wt{p}_k) \label{eq_update_formula_for_reflected_variables_x}\\
	\wt{H}_{k+1} & = V\bst{\bbr{I - \wt{\rho}_k \wt{s_k}\wt{y_k}^\intercal} \wt{H}_k \bbr{I - \wt{\rho}_k \wt{y_k}\wt{s_k}^\intercal} + \wt{\rho}_k \wt{s}_k \wt{s_k}^\intercal} V\label{eq_update_formula_for_reflected_variables_H}
\end{align}
where
\eqals{
	\wt{p}_k & = - \wt{H}_k \eta \\
	\wt{s}_k & = \alpha_k \wt{p}_k \\
	\wt{y}_k & = -2 \zeta\\
	\wt{\rho}_k & = \frac{1}{\wt{y}_k^\intercal \wt{s}_k},
}
and according to \eqref{eq_suff_dec} and \eqref{eq_curv} the step size $\alpha_k$ satisfies
\eqal{
	\label{eq_weak_wolfe}
	f(\wt{x}^k + \alpha_k \wt{p}^k) & \leq f(\wt{x}^k) + c_1 \alpha_k \bbr{\nabla f(\wt{x}^k)}^\intercal \wt{p}^k  \\
	\bbr{\nabla f(\wt{x}^k + \alpha_k \wt{p}^k)}^\intercal \wt{p}^k &\geq c_2 \bbr{\nabla f(\wt{x}^k)}^\intercal \wt{p}^k.
}

To simplify notation, for $u, v \in \mb{R}^n$ define the inner product
 $\bag{u, v}_k = u^\intercal \wt{H}_k v$.
Then we can write the update formula \eqref{eq_update_formula_for_reflected_variables_x} and \eqref{eq_update_formula_for_reflected_variables_H}  as
\begin{align}
	\wt{x}_{k+1} & = V(\wt{x}_k - \alpha_k \wt{H}_k \eta)\label{eq_update_formula_for_relfected_variables_compact_x} \\
	\wt{H}_{k+1} & = V\bst{\bbr{I - \frac{\wt{H}_k \eta \zeta^\intercal}{\bag{\eta, \zeta}_k}} \wt{H}_k \bbr{I - \frac{\zeta \eta^\intercal \wt{H}_k}{\bag{\eta, \zeta}_k}} + \frac{1}{2}\alpha_k \frac{\wt{H}_k \eta \eta^\intercal \wt{H}_k}{\bag{\eta, \zeta}_k}} V.\label{eq_update_formula_for_relfected_variables_compact_H}
\end{align}
Similarly, the line search conditions \eqref{eq_weak_wolfe} can be written in a compact form. The sufficient decrease condition becomes
\eqal{
	\label{eq_suff_dec_rewritten}
	\eta^\intercal \wt{x}_{k+1} &\leq \eta^\intercal \wt{x}_{k} - c_1 \alpha_k \bag{\eta, \eta}_k 
}
and the curvature condition becomes
\eqal{
	\label{eq_curv_rewritten}
	\zeta^\intercal \wt{x}_{k+1} & > 0 \\
	-\bag{\eta, \eta}_k + 2 \bag{\eta, \zeta}_k & \geq - c_2 \bag{\eta, \eta}_k.
}
The first inequality captures the fact that the curvature condition enforces $\idx{\wt{x}}{1}_{k}>0$ for all $k$.
Using the update formula \ref{eq_update_formula_for_reflected_variables_x} and \eqref{eq_Veta}, this is equivalent to
\eqal{
	\label{eq_suff_dec_rewritten_update}
	\alpha_k \bbr{2 \bag{\eta, \zeta}_k - \bpa{1-c_1}\bag{\eta, \eta}_k}  \leq 2 \zeta^\intercal \wt{x}_k
}
and
\begin{align}
	\label{eq_curv_rewritten_update_1}
	\alpha_k \bag{\eta, \zeta}_k & > \zeta^\intercal \wt{x}_k \\
	2 \bag{\eta, \zeta}_k & \geq (1- c_2) \bag{\eta, \eta}_k.\label{eq_curv_rewritten_update_2}
\end{align}
By looking at \eqref{eq_update_formula_for_relfected_variables_compact_x}-- \eqref{eq_curv_rewritten_update_2}, we observe that the key quantities determining the step size $\alpha_k$ are $\zeta^\intercal \wt{x}_k$, $\bag{\eta, \eta}_k$, and $\bag{\eta, \zeta}_k$. It turns out that we can keep track of these quantities with update formulas that involve only $\zeta^\intercal \wt{x}_k$, $\bag{\eta, \eta}_k$, $\bag{\eta, \zeta}_k$, and $\bag{\zeta, \zeta}_k$. These are given in the following theorem.
\begin{theorem}
\label{thm_iterates}

If $\bag{\eta, \zeta}_k>0$, we have the following iterative relationships:
\begin{align}
	\bag{\eta, \eta}_{k+1} &= \bbr{\frac{\bag{\eta, \eta}_k \bag{\zeta, \zeta}_k}{\bag{\eta, \zeta}_k^2} - 1}\bag{\eta, \eta}_k + \frac{1}{2} \alpha_k \bbr{\frac{\bag{\eta, \eta}_k^2}{\bag{\eta, \zeta}_k} - 4 \bag{\eta, \eta}_k + 4 \bag{\eta, \zeta}_k} \label{eq_update_formula_for_compact_variables_ee}\\
	\bag{\eta, \zeta}_{k+1} & = \alpha_k \bpa{\bag{\eta, \zeta}_k - \frac{1}{2}\bag{\eta, \eta}_k} \label{eq_update_formula_for_compact_variables_ez}\\
	\bag{\zeta, \zeta}_{k+1} & = \frac{1}{2} \alpha_k \bag{\eta, \zeta}_k\label{eq_update_formula_for_compact_variables_zz}\\
	\zeta^\intercal \wt{x}_{k+1} & = - \zeta^\intercal \wt{x}_k + \alpha_k \bag{\eta, \zeta}_k\label{eq_update_formula_for_compact_variables_zx}.
\end{align}
Furthermore, if a step size $\alpha_k$ exists in iteration $k$, the conditions \eqref{eq_suff_dec_rewritten_update}, \eqref{eq_curv_rewritten_update_1}, and \eqref{eq_curv_rewritten_update_2} must be satisfied.
\end{theorem}
\begin{proof}[Proof of theorem \ref{thm_iterates}]

For \eqref{eq_update_formula_for_compact_variables_zz}, we have from \eqref{eq_Veta} and \eqref{eq_update_formula_for_relfected_variables_compact_H}
\eqals{
	\bag{\zeta, \zeta}_{k+1} & = \zeta^\intercal V\bst{\bbr{I - \frac{\wt{H}_k \eta \zeta^\intercal}{\bag{\eta, \zeta}_k}} \wt{H}_k \bbr{I - \frac{\zeta \eta^\intercal \wt{H}_k}{\bag{\eta, \zeta}_k}} + \frac{1}{2}\alpha_k \frac{\wt{H}_k \eta \eta^\intercal \wt{H}_k}{\bag{\eta, \zeta}_k}} V \zeta \\
	& = \zeta^\intercal \bst{\bbr{I - \frac{\wt{H}_k \eta \zeta^\intercal}{\bag{\eta, \zeta}_k}} \wt{H}_k \bbr{I - \frac{\zeta \eta^\intercal \wt{H}_k}{\bag{\eta, \zeta}_k}} + \frac{1}{2}\alpha_k \frac{\wt{H}_k \eta \eta^\intercal \wt{H}_k}{\bag{\eta, \zeta}_k}} \zeta.
}
Note that
\eqal{\label{eq_zImH}
	\zeta^\intercal\bbr{I - \frac{\wt{H}_k \eta \zeta^\intercal}{\bag{\eta, \zeta}_k}} =  {\zeta^\intercal - \frac{\zeta^\intercal \wt{H}_k \eta \zeta^\intercal}{\bag{\eta, \zeta}_k}} 
	=  {\zeta^\intercal - \frac{\bag{\eta, \zeta}_k \zeta^\intercal}{\bag{\eta, \zeta}_k}} 
	=  {\zeta^\intercal - \zeta^\intercal} = 0
}
Thus,
\eqals{
	\bag{\zeta, \zeta}_{k+1} & = \zeta^\intercal \bst{\bbr{I - \frac{\wt{H}_k \eta \zeta^\intercal}{\bag{\eta, \zeta}_k}} \wt{H}_k \bbr{I - \frac{\zeta \eta^\intercal \wt{H}_k}{\bag{\eta, \zeta}_k}} + \frac{1}{2}\alpha_k \frac{\wt{H}_k \eta \eta^\intercal \wt{H}_k}{\bag{\eta, \zeta}_k}} \zeta \\
	& = \zeta^\intercal \bpa{\frac{1}{2}\alpha_k \frac{\wt{H}_k \eta \eta^\intercal \wt{H}_k}{\bag{\eta, \zeta}_k}} \zeta  = \frac{1}{2} \alpha_k \bag{\eta, \zeta}_k.
}

For \eqref{eq_update_formula_for_compact_variables_ez}, we have 
\eqals{
	\bag{\eta, \zeta}_{k+1} & \stackrel{\eqref{eq_update_formula_for_relfected_variables_compact_H}}{=} \zeta^\intercal V\bst{\bbr{I - \frac{\wt{H}_k \eta \zeta^\intercal}{\bag{\eta, \zeta}_k}} \wt{H}_k \bbr{I - \frac{\zeta \eta^\intercal \wt{H}_k}{\bag{\eta, \zeta}_k}} + \frac{1}{2}\alpha_k \frac{\wt{H}_k \eta \eta^\intercal \wt{H}_k}{\bag{\eta, \zeta}_k}} V \eta \\
	& \stackrel{\eqref{eq_Veta}}{=} - \zeta^\intercal \bst{\bbr{I - \frac{\wt{H}_k \eta \zeta^\intercal}{\bag{\eta, \zeta}_k}} \wt{H}_k \bbr{I - \frac{\zeta \eta^\intercal \wt{H}_k}{\bag{\eta, \zeta}_k}} + \frac{1}{2}\alpha_k \frac{\wt{H}_k \eta \eta^\intercal \wt{H}_k}{\bag{\eta, \zeta}_k}} (\eta - 2 \zeta) \\
	& \stackrel{\eqref{eq_zImH}}{=} - \zeta^\intercal \bpa{\frac{1}{2}\alpha_k \frac{\wt{H}_k \eta \eta^\intercal \wt{H}_k}{\bag{\eta, \zeta}_k}} (\eta - 2 \zeta) \\
	& = \frac{1}{2} \alpha_k \bbr{-\frac{\bag{\eta, \zeta}_k \bag{\eta, \eta}_k}{\bag{\eta, \zeta}_k} + 2 \frac{\bag{\eta, \zeta}_k^2}{\bag{\eta, \zeta}_k}} \\
	& = \alpha_k \bpa{\bag{\eta, \zeta}_k - \frac{\bag{\eta, \eta}_k}{2}}.
}

For \eqref{eq_update_formula_for_compact_variables_ee}, we have
\eqals{
	\bag{\eta, \eta}_{k+1} & \stackrel{\eqref{eq_update_formula_for_relfected_variables_compact_H}}{=} \eta^\intercal V\bst{\bbr{I - \frac{\wt{H}_k \eta \zeta^\intercal}{\bag{\eta, \zeta}_k}} \wt{H}_k \bbr{I - \frac{\zeta \eta^\intercal \wt{H}_k}{\bag{\eta, \zeta}_k}} + \frac{1}{2}\alpha_k \frac{\wt{H}_k \eta \eta^\intercal \wt{H}_k}{\bag{\eta, \zeta}_k}} V \eta \\
	&\stackrel{\eqref{eq_Veta}}{=} \bpa{\eta^\intercal - 2 \zeta^\intercal} \bst{\bbr{I - \frac{\wt{H}_k \eta \zeta^\intercal}{\bag{\eta, \zeta}_k}} \wt{H}_k \bbr{I - \frac{\zeta \eta^\intercal \wt{H}_k}{\bag{\eta, \zeta}_k}} + \frac{1}{2}\alpha_k \frac{\wt{H}_k \eta \eta^\intercal \wt{H}_k}{\bag{\eta, \zeta}_k}} \bpa{\eta -2 \zeta}\\
	& \stackrel{\eqref{eq_zImH}}{=}  \eta^\intercal\bbr{I - \frac{\wt{H}_k \eta \zeta^\intercal}{\bag{\eta, \zeta}_k}} \wt{H}_k \bbr{I - \frac{\zeta \eta^\intercal \wt{H}_k}{\bag{\eta, \zeta}_k}}\eta + \frac{1}{2} \alpha_k  \bpa{\eta^\intercal - 2 \zeta^\intercal} \frac{\wt{H}_k \eta \eta^\intercal \wt{H}_k}{\bag{\eta, \zeta}_k} \bpa{\eta -2 \zeta} \\
	& = \bbr{\frac{\bag{\eta, \eta}_k \bag{\zeta, \zeta}_k}{\bag{\eta, \zeta}_k^2} - 1}\bag{\eta, \eta}_k + \frac{1}{2} \alpha_k \bbr{\frac{\bag{\eta, \eta}_k^2}{\bag{\eta, \zeta}_k} - 4 \bag{\eta, \eta}_k + 4 \bag{\eta, \zeta}_k}.
}
%

For \eqref{eq_update_formula_for_compact_variables_zx}, we have
\eqals{
	\zeta^\intercal \wt{x}_{k+1} \stackrel{\eqref{eq_update_formula_for_relfected_variables_compact_x}}{=} \zeta^\intercal V (\wt{x}_k - \alpha_k \wt{H}_k \eta)
	\stackrel{\eqref{eq_Veta}}{=} - \zeta^\intercal (\wt{x}_k - \alpha_k \wt{H}_k \eta) = - \zeta^\intercal \wt{x}_k + \alpha_k \bag{\eta, \zeta}_k.
}

Finally, the constraints on $\alpha_k$ have been derived before the statement of the theorem. 
\end{proof}

We observe that the quantity $\bag{\zeta, \zeta}_k$ appears on the right-hand side of the update formulas in Theorem~\ref{thm_iterates} only in the form $\bag{\eta, \eta}_k \bag{\zeta, \zeta}_k - \bag{\eta, \zeta}_k^2$. Therefore, it is more convenient to define 
\eqals{
	D_k = \bag{\eta, \eta}_k \bag{\zeta, \zeta}_k - \bag{\eta, \zeta}_k^2
}
and rewrite \eqref{eq_update_formula_for_compact_variables_ee} and \eqref{eq_update_formula_for_compact_variables_zz} as
\begin{align}
\bag{\eta, \eta}_{k+1} &= \frac{D_k}{\bag{\eta, \zeta}_k^2} \bag{\eta, \eta}_k + \frac{1}{2} \alpha_k \frac{1}{\bag{\eta, \zeta}_k} \bpa{\bag{\eta, \eta}_k - 2\bag{\eta, \zeta}_k}^2 \label{eq_update_formula_for_compact_variables_ee_D} \\
	D_{k+1} & = \frac{1}{2}\alpha_k \frac{\bag{\eta, \eta}_k}{\bag{\eta, \zeta}_k} D_k, \label{eq_update_formula_for_compact_variables_D}
\end{align}
which greatly simplifies the iterative relations. 
Since $\tilde H_k$ is positive definite and $\eta\neq\zeta$, the Cauchy-Schwartz inequality implies that
\eqal{\label{eq_D_positive}
D_k > 0.
}

\section{Main result}\label{sec:mainresults}

We wish to show that the algorithm must terminate in finite steps for any choices of $c_1, c_2$ such that $0 < c_1 < c_2 < 1$. We can prove the following theorem.

\begin{theorem}[Main result]\label{thm:main}
Consider the standard BFGS method using the Armijo-Wolfe line search with parameters $0 < c_1 < c_2 < 1$ (referred to as ``the algorithm'') applied to the function $f$ defined in \eqref{eq_problem_definition}. Assume the step sizes are chosen in a way such that Assumption \ref{assumption_non_smooth_points_always_avoided} is satisfied. Then the algorithm terminates in finitely many iterations.
\end{theorem}

First, we develop a lower bound on the step size $\alpha_k$. Define
\eqals{
	\alpha_k^* = \frac{D_k}{\bag{\eta,\zeta}_k^2}\Bigg/\bpa{1 - \frac{1}{2} \frac{\bag{\eta, \eta}_k}{\bag{\eta, \zeta}_k}}.
}
We prove the following lemma:
\begin{lemma}
\label{lemma_lower_bound_on_alpha}
Suppose the algorithm does not terminate up to iteration $K \geq 2$. Then, we have
\begin{align}
	\bag{\eta, \zeta}_k > 0 \cm &\qquad \forall k \in [K],  \label{lemma_lower_bound_on_alpha_1} \\
	\bag{\eta, \zeta}_{k} - \frac{1}{2}\bag{\eta, \eta}_{k} > 0 \cm &\qquad \forall k \in [K-1], \label{lemma_lower_bound_on_alpha_2}\\
        \alpha_k^* > 0, &\qquad \forall k \in [K-1], \text{ and } \label{lemma_lower_bound_on_alpha_star}\\
	 \alpha_k > \frac{D_k}{\bag{\eta,\zeta}_k^2}\Bigg/\bpa{1 - \frac{1}{2} \frac{\bag{\eta, \eta}_k}{\bag{\eta, \zeta}_k}} \cm &\qquad \forall k \in [K-2].\label{lemma_lower_bound_on_alpha_3}
\end{align}
\end{lemma}

\begin{proof}[Proof of Lemma \ref{lemma_lower_bound_on_alpha}]
Since the algorithm does not break down up to iteration $K$, the step size $\alpha_k$ exists and satisfies the Armijo-Wolfe condition for all $k \in [K]$.
Theorem~\ref{thm_iterates} implies that \eqref{eq_curv_rewritten_update_2} holds, and therefore, for \textbf{any} choices of $c_1, c_2$ such that $0 < c_1 < c_2 < 1$, we must always have
\eqals{
	\bag{\eta, \zeta}_k > 0 \cm \forall k \in [K],
}
which proves \eqref{lemma_lower_bound_on_alpha_1}.

This also means that $\bag{\eta, \zeta}_{k+1} > 0$ for any $k \in [K-1]$, and then \eqref{eq_update_formula_for_compact_variables_ez} implies \eqref{lemma_lower_bound_on_alpha_2}, since $\alpha_k>0$.

Equation \eqref{lemma_lower_bound_on_alpha_2} in turn implies for any  $k \in [K-2]$, we have $\bag{\eta, \zeta}_{k+1} - \frac{1}{2}\bag{\eta, \eta}_{k+1} > 0$.  Substituting \eqref{eq_update_formula_for_compact_variables_ez} and \eqref{eq_update_formula_for_compact_variables_ee_D} yields
\eqals{
0 & < \alpha_k \bpa{\bag{\eta, \zeta}_k - \frac{1}{2}\bag{\eta, \eta}_k} - \frac{1}{2} \bst{\frac{D_k}{\bag{\eta, \zeta}_k^2} \bag{\eta, \eta}_k + \frac{1}{2} \alpha_k \frac{1}{\bag{\eta, \zeta}_k} \bpa{\bag{\eta, \eta}_k - 2\bag{\eta, \zeta}_k}^2}  \\
	& = \frac{1}{2}\bag{\eta, \eta}_k \bst{\bpa{1 - \frac{1}{2} \frac{\bag{\eta, \eta}_k}{\bag{\eta, \zeta}_k}} \alpha_k - \frac{D_k}{\bag{\eta,\zeta}_k^2}}.
}
Since $\bag{\eta, \eta}_k > 0$, it is
\eqals{
\bpa{1 - \frac{1}{2} \frac{\bag{\eta, \eta}_k}{\bag{\eta, \zeta}_k}} \alpha_k > \frac{D_k}{\bag{\eta,\zeta}_k^2}.
}
Finally, \eqref{lemma_lower_bound_on_alpha_2} implies that the coefficient before $\alpha_k$ is positive. Therefore, \eqref{lemma_lower_bound_on_alpha_3} holds.
\end{proof}

By Lemma \ref{lemma_lower_bound_on_alpha}, we obtained a lower bound for the step size $\alpha_k$, \textit {if the algorithm does not terminate} at iteration $k+2$.
%
We emphasize that the lower bound for $\alpha_k$ is \emph{not} a consequence of the Armijo-Wolfe condition in the \emph{current} iteration $k$.  Instead, it only must hold when the algorithm will not terminate two iterations \emph{later}.  This fact is important for establishing our main result:  We will show that at some iteration the Armijo-Wolfe conditions imply that $\alpha_k \leq \alpha_k^*$.  This then means that the algorithm will terminate at most 2 iterations later.


To this end, the next lemma establishes bounds that are implied by the Armijo-Wolfe conditions in the \emph{current} iteration.
\begin{lemma}
\label{lemma_bounds_on_alpha_Wolfe}
Suppose the algorithm does not terminate up to iteration $K$. Then, we have
\eqal{
	\label{eq_constraints_alpha_k}
	\frac{\zeta^\intercal \wt{x}_k}{\bag{\eta, \zeta}_k} < \alpha_k < \frac{{\zeta^\intercal \wt{x}_k}}{\bag{\eta, \zeta}_k - \frac{1}{2} \bag{\eta, \eta}_k} \cm \forall k \in [K-1].
}
\end{lemma}

\begin{proof}[Proof of Lemma \ref{lemma_bounds_on_alpha_Wolfe}]
Consider any $k \in [K-1]$. 
Then the sufficient decrease condition \eqref{eq_suff_dec} yields
\eqals{
\eta^\intercal \wt{x}_k & \stackrel{\eqref{eq_suff_dec_rewritten}}{>}\eta^\intercal \wt{x}_{k+1} \\
	& \stackrel{\eqref{eq_update_formula_for_reflected_variables_x}}{=}\eta^\intercal V (\wt{x}_k -\alpha_k \wt{H}_k \eta)  \\
	& \stackrel{\eqref{eq_Veta}}{=}(\eta^\intercal - 2\zeta^\intercal) (\wt{x}_k -\alpha_k \wt{H}_k \eta) \\
	& = \eta^\intercal \wt{x}_k - 2 {\zeta^\intercal \wt{x}_k} - \alpha_k \bag{\eta - 2 \zeta, \eta}_k .
}
This implies $\alpha_k \bag{2 \zeta - \eta, \eta}_k < 2 {\zeta^\intercal \wt{x}_k} $, and with \eqref{lemma_lower_bound_on_alpha_2} we obtain the upper bound in \eqref{eq_constraints_alpha_k}.

	

The other bound is a consequence of the curvature condition \eqref{eq_curv}.  Lemma~\ref{lemma_lower_bound_on_alpha} yields $\bag{\eta, \zeta}_k > 0$, and then Theorem~\ref{thm_iterates} implies \eqref{eq_curv_rewritten_update_1}.  This gives the lower bound in \eqref{eq_constraints_alpha_k}.

%
\end{proof}

%

To further simplify the iterative relationship, we define the following new quantities:
\begin{align}
	\delta_k & = 1 - \frac{\bag{\eta, \eta}_k}{2\bag{\eta, \zeta}_k} \label{eq_def_delta} \\
	a_k & = \frac{\alpha_k}{\alpha_k^*} = \frac{\delta_k }{D_k}\bag{\eta, \zeta}_k^2 \alpha_k  \label{eq_def_a} \\
	\psi_k & = \frac{\bag{\eta, \zeta}_k \zeta^\intercal \wt{x}_k}{D_k} \label{eq_def_psi}\\
	\gamma_k & = \frac{\zeta^\intercal\wt{x}_k}{\bag{\eta, \zeta}_k} \label{eq_def_gamma}
\end{align}
We have the following lemma.
\begin{lemma}
\label{lemma_simple_bounds_for_new_variables}
Suppose the algorithm does not terminate up to iteration $K \geq 2$. Then we have for all $k \in [K-2]$ that
\eqals{
	\delta_k & \in (0, 1) \\
	a_k & > 0 \\
	\psi_k & > 0 \\
	\gamma_k & > 0.
}
\end{lemma}
\begin{proof}[Proof of Lemma \ref{lemma_simple_bounds_for_new_variables}]
These conditions can be derived easily from \eqref{eq_curv_rewritten}, \eqref{eq_D_positive}, \eqref{lemma_lower_bound_on_alpha_1}, \eqref{lemma_lower_bound_on_alpha_2}, and \eqref{lemma_lower_bound_on_alpha_star}.
\end{proof}

\begin{lemma}
\label{lemma_iterative_relationship_for_new_var}
Suppose the algorithm does not terminate up to iteration $K \geq 2$. The following iterative relations holds for all $k \in [K-2]$
\eqals{
	\psi_{k+1} & = \psi_k - \frac{\psi_k-a_k}{1-\delta_k}\\
	\gamma_{k+1} & = \frac{1}{\delta_k }-\frac{\psi_k }{a_k} \\
	\delta_{k+1} & = \bpa{1-\frac{1}{a_k}} (1-\delta_k) \\
        \delta_k \psi_k & < a_k < \psi_k \\
        a_k &> 1.
}
\end{lemma}
\begin{proof}[Proof of Lemma \ref{lemma_iterative_relationship_for_new_var}]
Assume $k \in [K-2]$. 
Consider $\psi$ first. We have
\eqals{
	\psi_{k+1} & \stackrel{\eqref{eq_def_psi}}{=} \frac{\bag{\eta, \zeta}_{k+1} \zeta^\intercal \wt{x}_{k+1}}{D_{k+1}} \\
	&  \stackrel{\eqref{eq_update_formula_for_compact_variables_ee_D},\eqref{eq_update_formula_for_compact_variables_ez},\eqref{eq_update_formula_for_compact_variables_zx}}{=} \frac{2 \langle \eta ,\zeta \rangle _k \left(\langle \eta ,\zeta \rangle _k-\frac{1}{2} \langle \eta ,\eta \rangle _k\right) \left(\alpha _k \langle \eta ,\zeta \rangle _k-\zeta^\intercal  \wt{x}_k\right)}{D_k \langle \eta ,\eta \rangle _k} \\
	& \stackrel{\eqref{eq_def_delta},\eqref{eq_def_a},\eqref{eq_def_psi},\eqref{eq_def_gamma}}{=} \frac{\psi _k \left(\langle \eta ,\zeta \rangle _k-\left(1-\delta _k\right) \langle \eta ,\zeta \rangle _k\right) \left(\frac{a_k \gamma _k \langle \eta ,\zeta \rangle _k}{\delta _k \psi _k}-\gamma _k \langle \eta ,\zeta \rangle _k\right)}{\gamma _k \left(1-\delta _k\right) \langle \eta ,\zeta \rangle _k^2} \\
	& = \frac{\delta _k \psi _k-a_k}{\delta _k-1} \\
	& = \psi_k - \frac{\psi_k-a_k}{1-\delta_k}\\
}

Then consider $\gamma$, we have
\eqals{
	\gamma_{k+1} & \stackrel{\eqref{eq_def_gamma}}{=} \frac{\zeta^\intercal  \wt{x}_{k+1}}{\langle \eta ,\zeta \rangle _{k+1}} \\
	& \stackrel{\eqref{eq_update_formula_for_compact_variables_ez},\eqref{eq_update_formula_for_compact_variables_zx}}{=} \frac{\alpha _k \langle \eta ,\zeta \rangle _k-\zeta^\intercal  \wt{x}_k}{\alpha _k \left(\langle \eta ,\zeta \rangle _k-\frac{1}{2} \langle \eta ,\eta \rangle _k\right)} \\
	& = \frac{\delta _k \psi _k \left(\frac{a_k \gamma _k \langle \eta ,\zeta \rangle _k}{\delta _k \psi _k}-\gamma _k \langle \eta ,\zeta \rangle _k\right)}{a_k \gamma _k \left(\langle \eta ,\zeta \rangle _k-\left(1-\delta _k\right) \langle \eta ,\zeta \rangle _k\right)} \\
	& = \frac{1}{\delta _k}-\frac{\psi _k}{a_k}
}

Finally, consider $\delta$, we have
\eqals{
	\delta_{k+1} & = 1-\frac{\langle \eta ,\eta \rangle _{k+1}}{2 \langle \eta ,\zeta \rangle _{k+1}} \\
	& = 1-\frac{\frac{D_k \langle \eta ,\eta \rangle _k}{\langle \eta ,\zeta \rangle _k^2}+\frac{\alpha _k \left(\langle \eta ,\eta \rangle _k-2 \langle \eta ,\zeta \rangle _k\right){}^2}{2 \langle \eta ,\zeta \rangle _k}}{2 \alpha _k \left(\langle \eta ,\zeta \rangle _k-\frac{1}{2} \langle \eta ,\eta \rangle _k\right)} \\
	& = 1-\frac{\delta _k \psi _k \left(\frac{a_k \gamma _k \left(2 \left(1-\delta _k\right) \langle \eta ,\zeta \rangle _k-2 \langle \eta ,\zeta \rangle _k\right){}^2}{2 \delta _k \psi _k \langle \eta ,\zeta \rangle _k}+\frac{2 \gamma _k \left(1-\delta _k\right) \langle \eta ,\zeta \rangle _k}{\psi _k}\right)}{2 a_k \gamma _k \left(\langle \eta ,\zeta \rangle _k-\left(1-\delta _k\right) \langle \eta ,\zeta \rangle _k\right)} \\
	& = \frac{a_k \left(-\delta _k\right)+a_k+\delta _k-1}{a_k} \\
	& = \bpa{1-\frac{1}{a_k}} (1-\delta_k)
}

For $a$, Lemma \ref{lemma_lower_bound_on_alpha} yields $\alpha_k > \alpha_k^*$. Since $a_k = \alpha_k / \alpha_k^*$, we immediately have $a_k > 1$ by \eqref{eq_def_a}.
Then, consider Lemma \ref{lemma_bounds_on_alpha_Wolfe}, we have
\eqals{
	\frac{\zeta^\intercal \wt{x}_k}{\bag{\eta, \zeta}_k} < \alpha_k < \frac{{\zeta^\intercal \wt{x}_k}}{\bag{\eta, \zeta}_k - \frac{1}{2} \bag{\eta, \eta}_k}
}
Note that
\eqals{
	\delta_k = \frac{\zeta^\intercal \wt{x}_k}{\bag{\eta, \zeta}_k} \Bigg / \frac{{\zeta^\intercal \wt{x}_k}}{\bag{\eta, \zeta}_k - \frac{1}{2} \bag{\eta, \eta}_k}
}
and that
\eqals{
	\psi_k = \frac{{\zeta^\intercal \wt{x}_k}}{\bag{\eta, \zeta}_k - \frac{1}{2} \bag{\eta, \eta}_k} \Bigg / \alpha_k^*
}
we immediately get 
\eqals{
	\delta_k \psi_k < a_k < \psi_k.
	}
\end{proof}

Lemma \ref{lemma_iterative_relationship_for_new_var} shows that if the algorithm does not terminate up to iteration $K$, then the ``normalized'' step size $a_k$ must satisfies
\eqal{\label{eq_squeeze_a}
	\max \bst{\delta_k \psi_k, 1}  < a_k < \psi_k \cm \forall k \in [K-2].
}
The key observation here is that this condition can only be satisfied if
\eqals{
	\psi_k > 1.
}
However, we note that
\eqals{
	\psi_{k+1} & = \psi_k - \frac{\psi_k-a_k}{1-\delta_k},
}
which shows that $\bst{\psi_k}$ is a decreasing sequence, and we can exploit this fact to show that the algorithm always terminates finitely.

Now we are ready to prove the main result. 
\begin{proof}[Proof of Theorem~\ref{thm:main}]
We prove by contradiction. \textbf{Suppose the algorithm does not terminate}, generating infinite sequence $\bst{\wt{x}_k\cm k = 1,2,3,...}$. Then, by Lemma \ref{lemma_iterative_relationship_for_new_var}, we have for all $k \geq 0$ that
\eqal{\label{eq_update_psi}
	\psi_{k+1} & = \psi_k -\frac{\psi_k -a_k}{1 - \delta_k}.
}
Since $a_k \in \bpa{\max\{1, \delta_k \psi_k\}, \psi_k}$ by \eqref{eq_squeeze_a}, and that $\delta_k \in (0,1)$ by Lemma \ref{lemma_simple_bounds_for_new_variables}, we know that
\eqals{
	\psi_{k+1} < \psi_k.
}
Also, since the algorithm does not terminate, we must always have $\psi_{k} > 1$ by \eqref{eq_squeeze_a}. Therefore, $\{\psi_k\}$ is a decreasing sequence that is bounded below by $1$. It then follows that there exists $\psi^* \geq 1$ such that
\eqals{
	\lim_{k \to \infty} \psi_k & = \psi^*.
}

Using \eqref{eq_update_psi}, we then have
\eqals{
	0 < \psi^*-\psi_0 =  \sum_{k=0}^\infty \left(\psi_{k+1}- \psi_k\right) = \sum_{k=0}^\infty \frac{\psi_k - a_k}{1-\delta_k} < \infty,
}
which implies
\eqals{
	0 < \sum_{k=0}^\infty \bpa{\psi_k - a_k} < \sum_{k=0}^\infty \frac{\psi_k - a_k}{1-\delta_k} < \infty.
}
Also, $\psi_k - a_k > 0$ by \eqref{eq_squeeze_a}. Therefore, it must follows that
\eqals{
	\lim_{k \to \infty} \bpa{\psi_k - a_k} = 0,
}
i.e.,
\eqal{\label{eq_lim_a_psi}
	\lim_{k \to \infty} a_k = \lim_{k \to \infty} \psi_k = \psi^*.
}
Therefore, if the algorithm does not terminate, $a_k$ would have the same limit as $\psi_k$. 

Now we need to show that $\delta_k$ also attains a limit. Let $\delta^*$ be the unique solution to
\eqals{
	\delta^* = \bpa{1-\frac{1}{\psi^*}}(1-\delta^*).
}
From Lemma~\ref{lemma_iterative_relationship_for_new_var} we have
\eqals{
	\delta_{k+1} & = \bpa{1-\frac{1}{a_k}}(1-\delta_k) \\
}
We have
\eqal{\label{eq_delta_kp1_star}
	 \abs{\delta_{k+1} - \delta^*} 
	= & \abs{\bpa{1-\frac{1}{a_k}}(1-\delta_k) - \bpa{1-\frac{1}{\psi^*}}(1-\delta^*)} \\
	= & \abs{\bpa{1-\frac{1}{a_k}}-\bpa{1-\frac{1}{a_k}}\delta_k - \bpa{1-\frac{1}{\psi^*}} +\bpa{1-\frac{1}{\psi^*}}\delta^*} \\
	= & \abs{\bpa{1-\frac{1}{a_k}}- \bpa{1-\frac{1}{\psi^*}} -\bpa{1-\frac{1}{a_k}}\delta_k +\bpa{1-\frac{1}{a_k}}\delta^* + \bpa{\frac{1}{a_k}-\frac{1}{\psi^*}}\delta^*} \\
	= & \abs{\bpa{\frac{1}{\psi^*}-\frac{1}{a_k}} -\bpa{1-\frac{1}{a_k}}\delta_k +\bpa{1-\frac{1}{a_k}}\delta^* + \bpa{\frac{1}{a_k}-\frac{1}{\psi^*}}\delta^*} \\
	= & \abs{\bpa{\frac{1}{\psi^*}-\frac{1}{a_k}}(1-\delta^*) -\bpa{1-\frac{1}{a_k}}(\delta_k-\delta^*)} \\
	< & (1-\delta^*)\abs{{\frac{1}{\psi^*}-\frac{1}{a_k}}} + \bpa{1-\frac{1}{a_k}}\abs{\delta_k-\delta^*} \\ 
	< & \abs{{\frac{1}{\psi^*}-\frac{1}{a_k}}} + \bpa{1-\frac{1}{\psi_0}}\abs{\delta_k-\delta^*} ,
}
where we used $\psi_0>\psi_k>a_k > 1$ and $\delta^*<1$ in the last inequality.
Let 
\eqals{
	\Delta_k = \abs{{\frac{1}{\psi^*}-\frac{1}{a_k}}}
}
and 
\eqals{
	\rho = \bpa{1-\frac{1}{\psi_0}} \in (0,1).
}
Inequality \eqref{eq_delta_kp1_star} can then be written as
\eqals{
	\abs{\delta_{k+1} - \delta^*} < \Delta_k + \rho \abs{\delta_k - \delta^*}. 
}
which means that
\eqals{
	\abs{\delta_{k+l} - \delta^*} & < \sum_{j=0}^{l-1} \Delta_{k+j} \rho^{l-1-j} + \rho^l \abs{\delta_k - \delta^*} \\
	& < \sum_{j=0}^{l-1} \Delta_{k+j} \rho^{l-1-j} + \rho^l,
}
where the last inequality comes from $\delta_k,\delta^*\in(0,1)$. Further,
\eqals{ 
	\sum_{j=0}^{l-1} \Delta_{k+j} \rho^{l-1-j} + \rho^l 
	& < \max_{j \geq 0} \bst{\Delta_{k+j}} \sum_{j = 0}^\infty \rho^j + \rho^l \\
	& < \frac{1}{1-\rho}\max_{j \geq 0} \bst{\Delta_{k+j}} + \rho^l \\
	& = \frac{1}{1-\rho}\max_{j \geq k} \bst{\Delta_{j}} + \rho^l
}
Let
\eqals{
	m_k = \max_{j \geq k} \bst{\Delta_{j}}
}
then we have 
\eqal{
	\abs{\delta_{k+l} - \delta^*} & < \frac{1}{1-\rho}m_k + \rho^l
}
Note that $\Delta_k \to 0$ since $a_k \to \psi^* \geq 1$, implying that the sequence $m_k \to 0$ as well. Therefore, for any $\epsilon > 0$, we can find a $\overline{k} > 0$ such that $m_k < (1-\rho)\epsilon/2$ whenever $k \geq \overline{k}$; we can also find an $\overline{l}$ such that $\rho^l < \epsilon/2$ whenever $l \geq \overline{l}$. Then, whenever $n \geq \overline{k} + \overline{l}$, we have $\abs{\delta_n - \delta^*} < \epsilon$. This implies that 
\eqals{
	\lim_{k \to \infty} \abs{\delta_k - \delta^*} = 0.
}
As a final remark about $\delta_k$, we note that
\eqal{\label{eq_delta_star}
	\delta^* = \frac{1-1/\psi^*}{2-1/\psi^*} \in \left[0, \frac{1}{2}\right).
}
From the iterative relationship (see Lemma~\ref{lemma_iterative_relationship_for_new_var})
\eqals{
	\gamma_{k+1} = \frac{1}{\delta_k }-\frac{\psi_k }{a_k},
}
using \eqref{eq_lim_a_psi} we immediately see that the limit of $\gamma_k$ also exists (possibly infinite), with
\eqals{
	\gamma^* = \lim_{k \to \infty} \gamma_k & = \frac{1}{\delta^*} - 1 \in (1, +\infty].
}

Now we are ready to establish the contradiction. From the analysis above, we see that if we assume the algorithm does not terminate, we would deduce that $a_k$ goes to $\psi_k$ asymptotically. Recall that the condition
\eqals{
	a_k < \psi_k
}
comes from the requirement that the function value must be decreased at each iteration:
\eqals{
	f(\wt{x}_{k+1}) < f(\wt{x}_k).
}
This means that we are gradually taking steps with \textbf{very little decrease} in function value, and this should be prohibited by the sufficient decrease condition for $c_1 > 0$. 

Indeed, by the sufficient decrease condition in \eqref{eq_suff_dec_rewritten_update},
\eqals{
\alpha_k \bbr{2 \bag{\eta, \zeta}_k - \bpa{1-c_1}\bag{\eta, \eta}_k}  < 2 \zeta^\intercal \wt{x}_k
}
using the fact that $\bag{\eta, \zeta}_k > 0$ and $\gamma_k > 0$ as well as \eqref{eq_def_delta}--\eqref{eq_def_gamma}, we have 
\eqals{
 2 \gamma _k \langle \eta ,\zeta \rangle _k  = 2 \zeta^\intercal \wt{x}_k 
 	& < \frac{2 a_k \gamma _k \left(\delta _k-c_1 \left(\delta _k-1\right)\right) \langle \eta ,\zeta \rangle _k}{\delta _k \psi _k},
}
and so
\eqals{
1 < \frac{a_k \left(\delta _k-c_1 \left(\delta _k-1\right)\right)}{\delta _k \psi _k} =  \frac{c_1 a_k \left(1-\delta _k\right)}{\delta _k \psi _k}+\frac{a_k }{\psi _k},
}
which finally gives
\eqals{
\bpa{1-\frac{a_k}{\psi_k}}  > \frac{a_k  (1-\delta_k)}{\delta_k \psi_k} c_1
}
for all $k$.
However, this is impossible, since the left-hand side has a limit of $0$, while the right-hand side has a limit that is strictly larger than $0$ (see \eqref{eq_lim_a_psi} and \eqref{eq_delta_star}). Therefore, there is a contradiction, and thus the algorithm always terminates as long as we require $c_1 > 0$. This completes the proof for the main result.
\end{proof}

\begin{remark}
We point out that since $p_k$ is always a descend direction, we can always find a step size that satisfies the sufficient decrease condition. Therefore, whenever the BFGS terminates, the cause must be that it cannot find a step size that satisfies the curvature condition. When this happens, the function must be unbounded along the search direction, which means that the algorithm actually drives the function value to $-\infty$ in this case. Thus, the analysis above show that BFGS behaves in the desired way on this simple example in that it never gets stuck at a point where the function is not differentiable, and that it always drives the function value to $-\infty$. 
\end{remark}

\begin{remark}\label{rem_generalized}
We can generalize this observation to obtain a limited convergence result for \emph{any} piecewise linear function $g$:
If the unadulterated BFGS method converges to a limit point $x^*$ in a manifold of non-differentiable points with dimension 1 (i.e., only two linear pieces intersect), the limit point must be a local minimizer $g$; that is, $x^*$ cannot be a spurious solution.

This can be seen by the following argument.  If $\{x_k\}$ converges to $x^*$, then the iterates will eventually be contained in a neighborhood in which $g$ looks like a piecewise linear function with two pieces \eqref{eq_g_affine}, i.e., $g(x) = \abs{v_1^\intercal x} + v_2^\intercal x$.  If $v_1$ and $v_2$ were linearly independent, Theorems \ref{thm_equivalent} and \ref{thm:main} would imply that the algorithm cannot converge.  

Therefore, $v_1$ and $v_2$ must be linearly dependent, and we have $g(x) = |v_1^\intercal x| + \beta v_1^\intercal x$ for some $\beta\in\mathbb R$.  Theorem~\ref{thm_equivalent} then implies that the algorithm generates iterates that are equivalent to those generated for the objective function $f(x)=|\idx{x}{1}|+\beta \idx{x}{1}$ for a different initialization.   If $\abs{\beta} > 1$, the function is unbounded below. In this case, the algorithm terminates in at most two iterations with an unbounded search direction. To see this, assume without loss of generality that $\beta > 1$. Since the search direction $p_k$ is always a descent direction, we must have $\idx{p}{1}_k < 0$. If $\idx{x}{1}_0 < 0$, then the algorithm would terminate immediately because the gradient is constant along $\idx{p}{1}_0$ and the curvature condition cannot be satisfied for any step length. If $\idx{x}{1}_0 > 0$, then we must have $\idx{x}{1}_1 < 0$ due to the curvature condition.  Now the previous argument applies for the second iteration because $\idx{x}{1}_1 < 0$, and so no step size satisfying the curvature condition exists.

Hence, if $\{x_k\}$ converges to $x^*$, the only possible case is $\abs{\beta} \leq 1$.  Since $x^*$ is in a manifold of non-differentiable points with dimension 1, we must have $\idx{(x^*)}{1}=0$ and so $x^*$ is in the set of local minimizer for this specific function $f$.

\end{remark}

\section{Numerical Results}\label{sec:numerical}

In this section, we  provide some numerical results which verify the correctness of our analysis above. The computations were performed in Mathematica, with a precision of 500 digits. 
The line search algorithm used is the standard Armijo-Wolfe bracketing line search, with parameters $c_1 = 1 \times 10^{-4}$, $c_2 = 1/2$ as is used in \cite{lewis2013nonsmooth}. We sample the initial Hessian inverse approximation $H_0$ as
\eqals{
	H_0 = X^\intercal X \cm \idx{X}{i, j} \overset{\text{i.i.d}}{\sim} N(0, 1) 
}
and the initial point is sampled as 
\eqals{
	x_0(i) \overset{\text{i.i.d}}{\sim} N(0, 1).
}
The algorithm is terminated when the number of line search trials exceeds $1000$. 

To demonstrate the typical number of iterations before the algorithm terminates, we ran $5000$ experiments for each of the dimensions $n = 2, 3, 5, 10, 30,$ and $100$ respectively.  The outcomes are summarized in Table \ref{tab_num_of_iterations}. 

\begin{table}[h!]
\centering
\caption{\label{tab_num_of_iterations}Number of iterations before the algorithm terminates.}
\begin{tabular}{ c c c c c  } 
 \hline \hline
 $n$ & Min & Max & Mean & Median \\
 \hline 
 2 & 1. & 23. & 2.7862 & 2. \\
 3 & 1. & 11. & 1.8106 & 2. \\
 5 & 1. & 6. & 1.4034 & 1. \\
 10 & 1. & 2. & 1.0696 & 1. \\
 30 & 1. & 1. & 1. & 1. \\
 100 & 1. & 1. & 1. & 1. \\
 \hline \hline
\end{tabular}
\end{table}

As a numerical verification of our analysis as well as a demonstration of the idea of the proof, we take the iterates generated by the algorithm (with $(c_1, c_2) = (1\times10^{-4}, 1/2)$) for one particular instance of the experiment. The example run has dimension $n = 3$, and is initialized as
\eqals{
	x_0 & = (1.24613,0.974014,-2.07329)^\intercal \\
	H_0 & = \left(
			\begin{array}{ccc}
			 2.9778 & -0.614829 & -0.764638 \\
			 -0.614829 & 0.93846 & -0.699262 \\
			 -0.764638 & -0.699262 & 1.11173 \\
			\end{array}
			\right)
}
and terminates at iteration $k = 9$. Since it does not terminate in iteration $k=8$, for $k \in \bst{0,1,2,...,6}$, we have $\delta_k \psi_k < a_k < \psi_k$ by Lemma~\ref{lemma_iterative_relationship_for_new_var}. As a verification of our analysis, we plot the logarithm of the value $\delta_k \psi_k$, $a_k$, and $\psi_k$ for $k \in \bst{0,1,2,...,6}$, as shown in Figure \ref{fig_verification}. We  observe that indeed  $\delta_k \psi_k < a_k < \psi_k$, $\delta_k \in (0,1)$, $\psi_k > 1$, and that $\bst{\psi_k}$ decreases monotonically, as predicted by Lemma~\ref{lemma_iterative_relationship_for_new_var}.

In addition, we consider the same example, but now with line search parameters $(c_1, c_2) = (0, 1)$. This time the algorithm continues for two more iterations and terminates at iteration $k = 10$. Figure \ref{fig_verification_2} shows the corresponding quantities. This example is interesting in that we  observe that $\psi_k$  becomes less than $1$ (i.e. $\log \psi_k$ becomes negative) at $k = 8$. Our analysis above predicts that this results in the algorithm's termination at iteration $k = 8 + 2 = 10$, which is precisely what happens. These observations are consistent with our analysis above. 

\begin{figure}[!h]
	\centering
	\includegraphics[width=.8\textwidth]{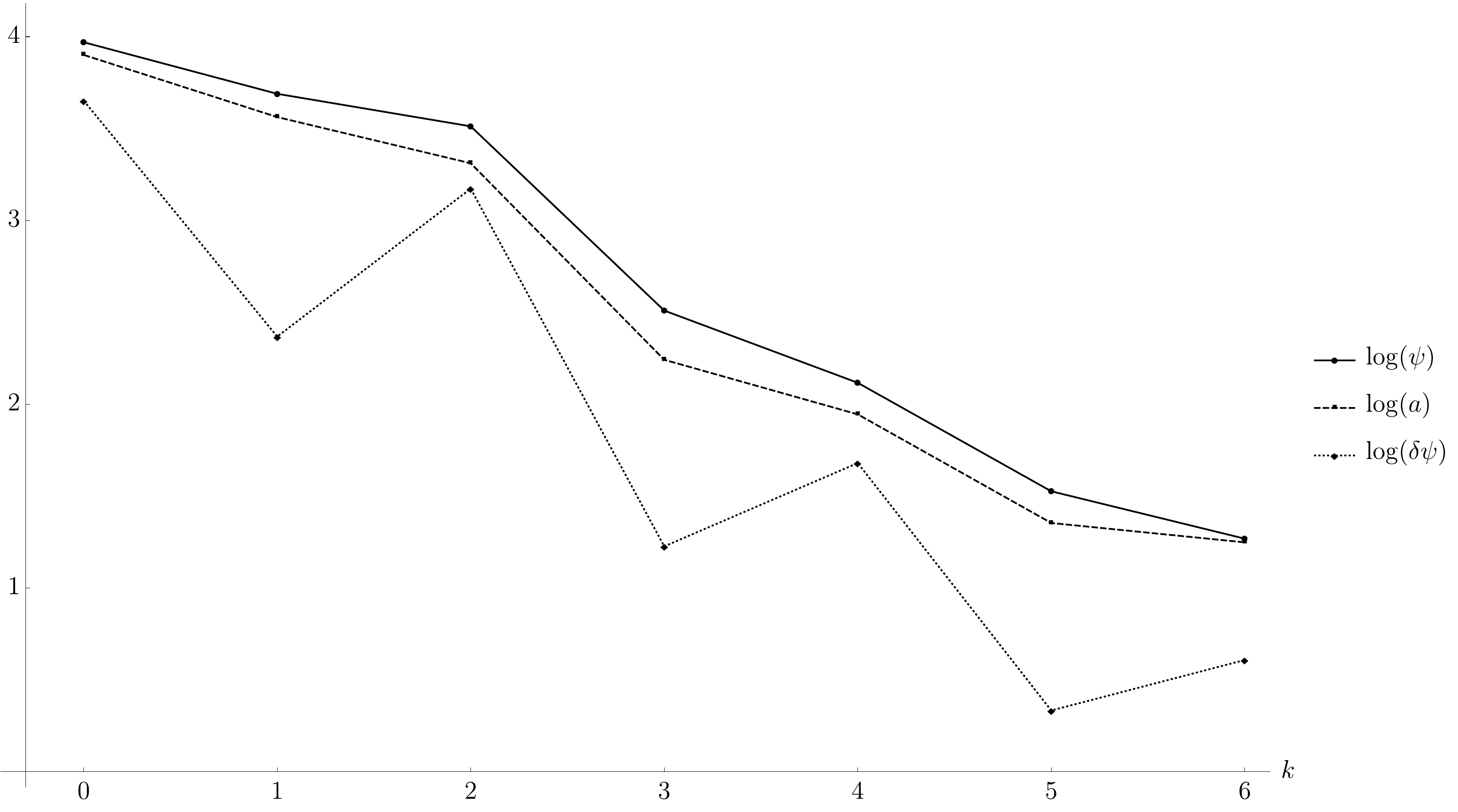}
	\caption{Value $(\log(\delta_k \psi_k), \log(a_k), \log(\psi_k))$ for iterates in one run of the experiment with $(c_1, c_2) = (1\times 10^{-4}, 1/2)$. }
	\label{fig_verification}
\end{figure}

\begin{figure}[!h]
	\centering
	\includegraphics[width=.8\textwidth]{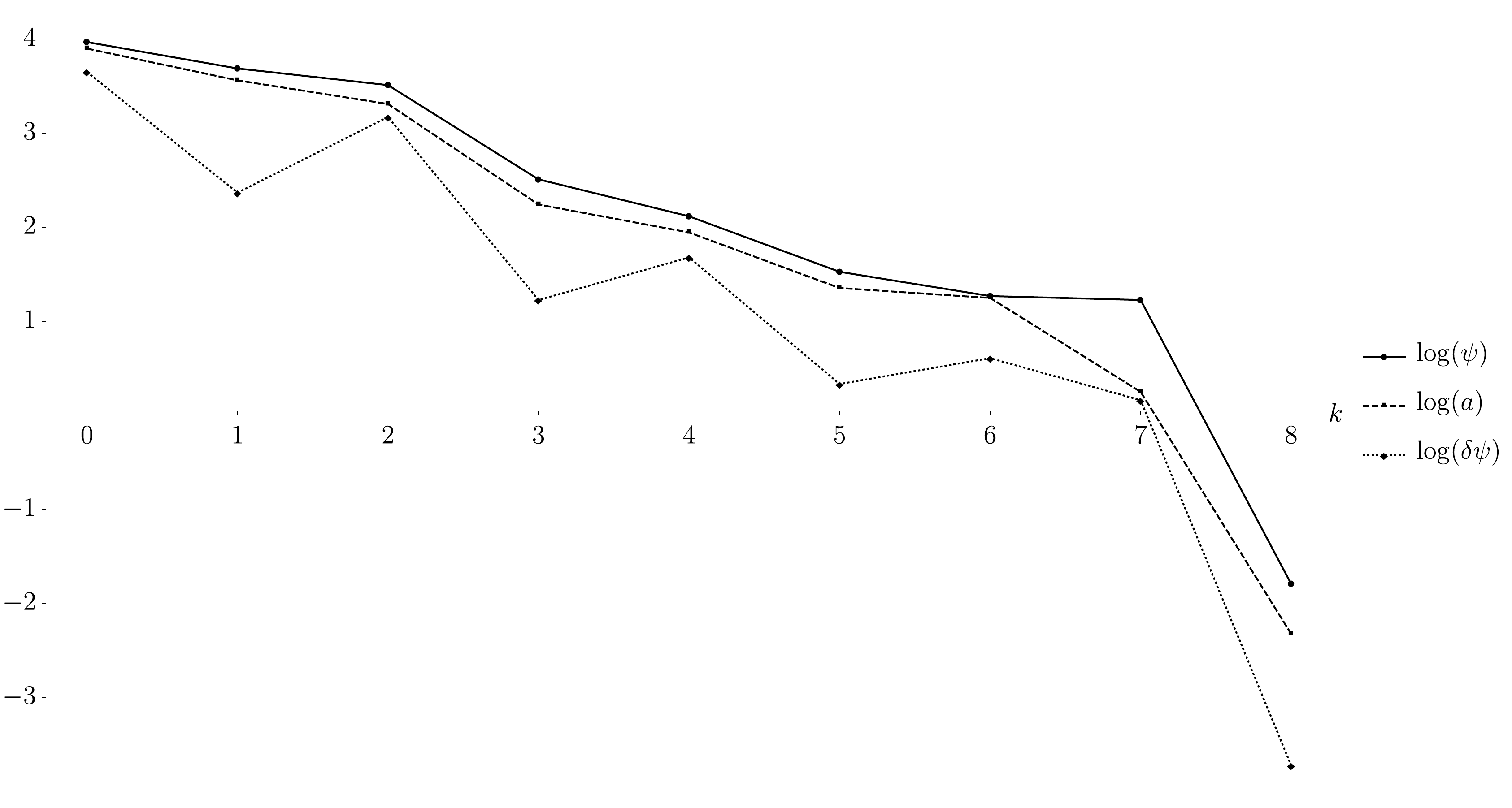}
	\caption{Value $(\log(\delta_k \psi_k), \log(a_k), \log(\psi_k))$ for iterates in one run of the experiment with $(c_1, c_2) = (0, 1)$. }
	\label{fig_verification_2}
\end{figure}

\section{Conclusion}\label{sec:conclude}
In this paper, we analyzed the behavior of the unadulterated BFGS method with an Armijo-Wolfe line search on a special family of non-smooth functions, namely, $g(x) = \abs{v_1^\intercal x} + v_2^\intercal x$ where $\bst{v_1, v_2}$ are linear independent. We have proved that the BFGS method always terminates in a finite number of iterations on this family of functions, driving the function values to $-\infty$.  We presented numerical results consistent with our theoretical analysis. Although limited to this special family, we hope that our analysis of BFGS will shed light on the behavior of the BFGS method on general non-smooth functions, specifically, on the reason why BFGS never converges to a point of non-differentiability that is not a stationary point. 

\newpage
\bibliographystyle{unsrt}
\bibliography{main}
\end{document}